\documentclass[a4paper,10pt]{article}
\usepackage{amsmath,amssymb,amsfonts,amsthm}
\usepackage{graphicx,color,subfigure}
\usepackage{enumerate,multirow}
\usepackage{fullpage}
\usepackage[]{hyperref}
\usepackage[normalem]{ulem}

\newcommand\RR{\ensuremath{\mathbb{R}}}

\newcommand\ZZ{\ensuremath{\mathbb{Z}}}
\newcommand\NN{\ensuremath{\mathbb{N}}}

\newcommand\Tor{\mathsf{T}}
\newcommand\Hex{\mathsf{H}}
\newcommand\Str{\mathsf{S}}

\newcommand\bH{b^{\Hex}}
\newcommand\bS{b^{\Str}}
\newcommand\Cont{\mathsf{C}}

\newtheorem{thm}{Theorem}[section]
\newtheorem{prop}[thm]{Proposition}
\newtheorem{cor}[thm]{Corollary}
\newtheorem{lem}[thm]{Lemma}

\newtheorem{conj}[thm]{Conjecture}

\title{Examples of spectral minimal partitions}
\author{Corentin L\'ena\footnote{Dipartimento di Matematica \emph{Giuseppe Peano}, Universit\`a degli Studi di Torino, Via Carlo Alberto, 10, 10123 Torino (TO), Italia, \texttt{clena@unito.it}}}
	
\begin{document}
	\maketitle
	
	\begin{abstract}
	
	We study a minimal partition problem on the flat rectangular torus. We give a partial review of the existing literature, and present some numerical and theoretical work recently published elsewhere by V. Bonnaillie-No{\"e}l and the author, with some improvements. 
		
	\end{abstract}
	
	\paragraph{Keywords.}   Minimal partitions, shape optimization,  Dirichlet-Laplacian eigenvalues, symmetrization, finite difference method, projected gradient algorithm.
	
	\paragraph{MSC classification.}  Primary 49Q10; Secondary 35J05, 65K10, 65N06, 65N25. 	
	
\section{Introduction}
\label{secIntro}
\subsection{Minimal partitions}	
The topic of spectral minimal partitions has been actively investigated by the shape optimization community during recent years. In addition to its intrinsic interest, it has many applications, for instance condensed matter physics, mathematical ecology or data sorting. In this review, we focus on one specific problem, for which the quantity to be optimized depends on the Dirichlet Laplacian eigenvalues. This problem is intimately connected with the nodal patterns of Laplacian eigenfunctions. Although we begin by recalling quite general results on minimal partitions in two dimensions, the paper then focus on the model problem of the flat rectangular torus. For the most part, we review the numerical and theoretical results obtained by the author in collaboration with V. Bonnaillie-No\"el in \cite{BonnaillieNoelLena2016ExpMath}. We also present a new lower bound on transition values which improves  existing estimates (Proposition \ref{propbksLower}). We point out that the authors previously studied circular sectors in a similar way \cite{BonnaillieNoelLena2014Sector}.

Let $\Omega$ be a bounded open set in  $\RR^2$ or in a $2$-dimensional Riemannian manifold. For any open subset $D$ of $\Omega$, let $ (\lambda_k(D))_{k\ge 1}$ be the eigenvalues of the Dirichlet Laplacian in $D$, arranged in non-decreasing order and counted with multiplicities. A \emph{$k$-partition} of $\Omega$ is a family $\mathcal D=(D_1,\dots,D_k)$ of open, connected and mutually disjoint subsets. We define its \emph{energy} as $\Lambda_{k}(\mathcal{D})=\max_{1\le i \le k}\lambda_1(D_i)$. A $k$-partition $\mathcal D^*$ is called \emph{minimal} if it has minimal energy, which we denote by $\mathfrak{L}_{k}(\Omega)$.

Let us introduce some additional notions, which enable us to describe the regularity of minimal partitions. We say that the $k$-partition $\mathcal D=(D_1,\dots,D_k)$ is \emph{strong} if it fills the set $\Omega$, that is to say if
\begin{equation*}
 \Omega=\mbox{Int}\left(\cup_{i=1}^k\overline D_i\right)\setminus \partial\Omega.
\end{equation*}
In that case, we define the \emph{boundary} of $\mathcal{D}$ as $N(\mathcal D):=\overline{\cup_{i=1}^k{\partial D_i}\setminus \partial \Omega}$. We say that $\mathcal D$ is \emph{regular} if it is strong $N(\mathcal D)$ satisfies the following properties.
\begin{enumerate}[i.]
	\item It is a union of regular arcs connecting a finite number of singular points (inside $\Omega$ or possibly on $\partial\Omega$).
	\item At the singular points, the arcs meet with equal angles (taking into account $\partial\Omega$ if necessary).
\end{enumerate}
Point ii is called the \emph{equal angle meeting property}. Let us note that these properties of $N(\mathcal D)$ are also satisfied by the nodal set of a Dirichlet Laplacian eigenfunction. However, in this latter case, the singular points inside $\Omega$ are crossing points, so the number of arcs meeting there must be even. This number can be odd in the case of a minimal partition. 

Existence and regularity of minimal partitions follow from the work of several authors: D. Bucur, G. Buttazzo, and A. Henrot \cite{BucButHen98};  L. Caffarelli and F.-H. Lin \cite{CafLin07}; M. Conti, S. Terracini, and G. Verzini \cite{ConTerVer05b}; B. Helffer, T. Hoffmann-Ostenhof, and S. Terracini \cite{HelHofTer09}. In the rest of the paper, we refer to the results by Helffer, Hoffmann-Ostenhof, and Terracini.
\begin{thm} Let $\Omega$ be a bounded open set in $\RR^2$ with a piecewise-$C^{1,+}$ boundary and satisfying the interior cone property. Then, for any positive integer $k$, 
\begin{enumerate}[i.]
	\item there exists a minimal $k$-partition of $\Omega$;
	\item any minimal $k$-partition of $\Omega$ is regular up to $0$-capacity sets.
\end{enumerate}
\end{thm}
Reference \cite{HelHofTer09} also establishes the  \emph{subpartition property}, which we use at the end of the present paper.
\begin{thm}
	Let $\mathcal D=(D_i)_{1\le i \le k}$ be a minimal $k$-partition of $\Omega$. Let $I\subset \{1,\dots, k\}$ with $k':=\sharp I$, $k'<k$, such that 
	\[\Omega_{I}:=\mbox{Int}\left(\bigcup_{i\in I}\overline{D}_i\right)\]
	is a connected open set. Then
	the sub-partition $\mathcal D_I=(D_i)_{i\in I}$ is the unique minimal $k'$-partition of $\Omega_I$ (up to $0$-capacity sets).
\end{thm}

\begin{cor}[pair compatibility condition] 
\label{corPCC}
	Let $\mathcal D=(D_i)_{1\le i \le k}$ ($k\ge3$) be a minimal $k$-partition of $\Omega$. For any two neighbors $D_i$ and $D_j$, the second eigenvalue of the Dirichlet Laplacian on 	$D_{i,j}:=\mbox{Int}\left(\overline D_i\cup \overline D_j\right)$,
	is simple, and $D_i$ and $D_j$ are the nodal domains of an eigenfunction associated with $\lambda_2(D_{i,j})$.
\end{cor}
\subsection{Nodal partitions}
If $u$ is an eigenfunction of the Dirichlet Laplacian in $\Omega$, the connected components of the complement of its zero set are called its nodal domains. Let us denote by $\nu(u)$ the number of nodal domain of $u$. The family $\mathcal{D}_u=(D_i)_{1\le i \le \nu(u)}$ of all the nodal domains of $u$ is the nodal partition associated with $u$. Given a regular $k$-partition $\mathcal{D}=(D_i)_{1\le i \le k}$, we say that two domains $D_i$ and $D_j$ are neighbors if they have a common boundary not reduced to points, that is to say if the set $D_{i,j}:=\mbox{Int}\left(\overline D_i\cup \overline D_j\right)$ is connected.

\begin{thm}
	A minimal $k$-partition  of $\Omega$ is nodal  if, and only if, it is bipartite, that is to say if we can color its domains with only two colors such that two neighbors have a different color.
\end{thm}

\begin{thm}[Courant, 1923]	
	If $u$ is an  eigenfunction  associated with $\lambda_k(\Omega)$, $\nu(u)\le k$.
\end{thm}

\begin{thm}[Courant-sharp characterization]
	\label{thmCourantSharp}
	The nodal partition associated with the eigenfunction $u$ is minimal if, and only if, $u$ is Courant-sharp, that is to say associated with $\lambda_k(\Omega)$, where $k=\nu(u)$.
\end{thm}

In particular, a  minimal $2$-partition is always the nodal  partition associated with a second eigenfunction. Theorem \ref{thmCourantSharp} allows one to give explicit examples of minimal partitions, in domains $\Omega$ for which the eigenvalues and eigenfunctions of the Laplacian are explicitly known: see for instance \cite{HelHofTer09,BonHelHof09}. Combined with topogical arguments and covering surface, it can also be use to produce example of non-nodal minimal partitions \cite{HelHofTer10a,HelHof14,SoaveTerracini2015}. Let us add that while minimal partitions are in general not nodal for the Dirichlet Laplacian \cite[Corollary 7.8]{HelHofTer09}, they are always nodal for a magnetic Laplacian, with a suitable magnetic potential of Aharonov-Bohm type, as was proved by B. Helffer and T. Hoffmann-Ostenhof \cite{HelHof13a} (see also \cite{NorTer10,BonHelHof09,HelHof15Review}).

\section{Transitions for the flat torus}
\label{secTrans}
\subsection{Statement of the problem}
Let us now describe our model problem. We consider the flat rectangular torus of length $a$ and width $b$: $\Tor(a,b)= \left(\RR/a\ZZ\right)\times \left(\RR/b\ZZ\right).$ The set of its eigenvalue is $\{\lambda_{m,n}(a,b)\,;\,(m,n)\in\NN_0^2\}$, with
\[\lambda_{m,n}(a,b)=4\pi^2\left(\frac{m^2}{a^2}+\frac{n^2}{b^2}\right),\]
and a corresponding basis of eigenfunctions is given by 
\[u_{m,n}^{a,b}(x,y)=\varphi\left(\frac{2m\pi x }a\right)\psi\left(\frac{2n\pi y}b\right),\] 
where $\varphi, \psi \in \{\cos,\sin\}$.

We first consider the partition of  $\Tor(a,b)$ into $k$ equal vertical strips: $\mathcal{D}_{k}(a,b)=(D_i)_{1\le i \le k}$, with
\[D_i=\left(\frac{i-1}{k}a,\frac{i}{k}a\right)\times\left(0,b\right).\]
Its energy is $\Lambda_k(\mathcal{D}_{k}(a,b))=k^2\pi^2/a^2$. We investigate the following question: for which values of $b\in(0,1]$ is $\mathcal D_k(1,b)$ a minimal partition of $\Tor(1,b)$. More specifically, let us define the transition value
\[ b_k=\sup\{b\in (0,1]~;~\mathcal{D}_{k}(1,b)\mbox{ is a minimal $k$-partition of } \Tor(1,b)\}.\]
The following result justifies the term \emph{transition value} (see \cite[Proposition 2.1]{BonnaillieNoelLena2016ExpMath}).
\begin{prop}
	The partition $\mathcal{D}_k(1,b)$ is minimal for all $b\in (0,b_k]$.
\end{prop}
We want to localize as precisely as possible this transition value. Let us first recall a result of Helffer and Hoffmann-Ostenhof \cite{HelHof14}.
\begin{thm}
	\label{thmTrans}
	If $k$ is even, $b_k=2/k$. If $k$ is odd, $b_k \ge 1/k$.
\end{thm}
We want to improve the lower bound when $k$ is odd. This can be done by considering the following auxiliary optimization problem. For $b\in (0,1]\,$, we consider the infinite strip $\Str_b=\RR\times \left(0,b\right)\,$ and we define
\begin{equation*}
\bS_k =\sup\left\{b\in(0,1]\,;\, j(b)>k^2\pi^2\right\},
\mbox{ with }
j(b)=\inf_{\Omega\subset \Str_b, |\Omega|\le b}\lambda_1(\Omega).
\end{equation*}
As seen in \cite[Theorem 1.9]{BonnaillieNoelLena2016ExpMath}, $b_k \ge \bS_k$ if $k$ is odd. The following estimate gives a quantitative improvement of Theorem \ref{thmTrans} and of \cite[Theorem 1.9]{BonnaillieNoelLena2016ExpMath}
\begin{prop}
	\label{propbksLower}
	For any integer $k\ge 2$, $1/\sqrt{k^2-1/8}\le\bS_k<1/\sqrt{k^2-1}$.
\end{prop}

As was pointed out to us by Bernard Helffer, the method of covering surfaces in \cite{HelHof14} leads quite naturally to the following conjecture.

\begin{conj}
	\label{conjbk}
	For any odd integer $k\ge 3$, $b_k=2/\sqrt{k^2-1}$.
\end{conj}

It can actually be proved that $b_k\le 2/\sqrt{k^2-1}$ (see \cite[Proposition 2.8]{BonnaillieNoelLena2016ExpMath}). The conjecture is supported by the numerical study. Proposition \ref{propbksLower} shows that, for any odd integer $k\ge 3$, $\bS_k<2/\sqrt{k^2-1}$. New idea would therefore be needed to prove Conjecture \ref{conjbk}.

\subsection{Proof of Proposition \ref{propbksLower}}

Let us sketch the proof of Proposition \ref{propbksLower}. It is a direct consequence of  the following proposition, after rescaling.
\begin{prop} \label{propJBounds} For $V\ge 1/2$,
	\[\pi^2\left(1+\frac1{8V^2}\right)\le J(V)< \mathcal \pi^2\left(1+\frac1{V^2}\right),\mbox{ where }J(V):=\inf_{\Omega\subset \Str_1, |\Omega|\le V}\lambda_1(\Omega).
	\]
\end{prop}
Let us note that in Proposition \ref{propJBounds}, and in the rest of this section, we define $\lambda_1(\Omega)$ for any open set in $\RR^2$, possibly unbounded and of infinite volume, as the infimum of a Rayleigh quotient:
\begin{equation*}
\lambda_1\left(\Omega\right):=\inf_{u\in H^1_0\left(\Omega\right)\setminus\{0\}}\frac{\int_{\Omega}\left|\nabla u\right|^2\,dx}{\int_{\Omega}u^2\,dx}.
\end{equation*}

The upper bound of $J(V)$ is obtained immediately by considering the rectangle  $(0,V)\times (0,1)$, which cannot be minimal, since the normal derivative of the first eigenfunction on its free boundary is not constant. The lower  bound is harder to prove. The first part of the proof relies on a symmetrization argument. For all $V>0$, we define an open subset $\Cont_V$ of $\Str$ by
\begin{equation*}
\Cont_V:=\left\{(x_1,x_1)\in\RR^2\,:\, \left|x_2-\frac12\right|< g(x_1) \right\}\mbox{ with }g(x_1):=\min\left(\frac12,\frac{V}{4x_1}\right).
\end{equation*}

\begin{lem}
	\label{lemJLower}
	For all $V>0$, $J(V)\ge \lambda_1\left(\Cont_V\right)$.
\end{lem}

\begin{proof}
	Let $\Omega$ be an open subset of $\Str$, of volume $V$. We perform two successive  Steiner symmetrizations, with respect to the lines $x_1=0$ and $x_2=\frac12$, and denote by $\Omega^*$ the resulting set. We have, according to the definition of Steiner symmetrization,
	\begin{equation*}
	\Omega^*=\left\{(x_1,x_2)\in\RR^2\,:\, \left|x_2-\frac12\right|<f(x_1)\right\},
	\end{equation*}
	where $f:\RR\to \left[0,\frac12\right]$ is an even function, non-increasing in $[0,+\infty)$. Since $f$ is non-increasing, we have, for all $x_1\in (0,+\infty)$,
	\begin{equation*}
	x_1f(x_1)\le \int_{0}^{x_1} f(t)\,dt\le \int_{0}^{+\infty} f(t)\,dt =\frac{V}4,	
	\end{equation*}
	and therefore	$f(x_1)\le \frac{V}{4x_1}$.
    This implies that $f(x_1)\le g(x_1)$, and therefore $\Omega^*\subset\Cont_V$. Since the first Dirichlet Laplacian eigenvalue is non-increasing with respect to Steiner symmetrization and the inclusion of domains, we obtain
	\begin{equation*}
	\lambda_1\left(\Cont_V\right)\le \lambda_1\left(\Omega^*\right)\le \lambda_1(\Omega).
	\end{equation*}
	Passing to the infimum, we get the desired result.
\end{proof}
To conclude the proof of Proposition \ref{propJBounds}, we obtain an explicit lower bound of $\lambda_1\left(\Cont_V\right)$. For $h>0$, let us define the ordinary differential operator $P_h$ by
\begin{equation*}
P_h:=-h^2\frac{d^2}{dt^2}+\pi^2(t^2-1)_+,
\end{equation*}
with $(t^2-1)_+:=\max\left(0,t^2-1\right)$. This operator is positive and self-adjoint, with compact resolvent. It therefore has discrete spectrum, and we denote by $\mu_1(h)$ its first eigenvalue.
\begin{lem} \label{lemL1Lower} For all $V>0$, $\lambda_1\left(\Cont_V\right)\ge \pi^2+\mu_1\left(\frac2V\right)$.
\end{lem}

\begin{proof}
	Let $u$ be a smooth function compactly supported in $\Cont_V$. We have 
	\begin{equation*}
	\int_{\Cont_V}\left|\nabla u \right|^2\,dx= \int_{-\infty}^{+\infty}dx_1\int_{\frac12-g(x_1)}^{\frac12+g(x_1)}dx_2\left(\left|\partial_{x_1}u\right|^2+\left|\partial_{x_2}u\right|^2\right).
	\end{equation*}
	For a given $x_1$, the one-dimensional Poincar\'e inequality on the segment $\left(1/2-g(x_1),1/2+g(x_1)\right)$ gives us
	\begin{equation*}
	\int_{\frac12-g(x_1)}^{\frac12+g(x_1)}\left|\partial_{x_2}u\right|^2\,dx_2\ge \frac{\pi^2}{4g(x_1)^2}\int_{\frac12-g(x_1)}^{\frac12+g(x_1)}u^2\,dx_2.
	\end{equation*}
	We obtain therefore
	\begin{equation*}
	\int_{\Cont_V}\left|\nabla u \right|^2\,dx\ge \int_{\frac12-g(x_1)}^{\frac12+g(x_1)}dx_2\int_{-\infty}^{+\infty}dx_1\left(\left|\partial_{x_1}u\right|^2+\frac{\pi^2}{4g(x_1)^2}u^2\right).
	\end{equation*}
	We now denote by $\nu_1(V)$ the first eigenvalue of the ordinary differential operator
	\begin{equation*}
	Q_V:=-\frac{d^2}{dx_1^2}+\frac{\pi^2}{4g(x_1)^2}.
	\end{equation*}
	According to the variational characterization of $\nu_1(V)$, we get
	\begin{equation*}
	\int_{-\infty}^{+\infty}dx_1\left(\left|\partial_{x_1}u\right|^2+\frac{\pi^2}{4g(x_1)^2}u^2\right)\ge \nu_1(V)\int_{-\infty}^{+\infty}u^2\,dx_1
	\end{equation*} 
	for all $x_2 \in (0,1)$, and therefore
	\begin{equation*}
	\int_{\Cont_V}\left|\nabla u \right|^2\,dx\ge \nu_1(V)\int_{\Cont_V}u^2\,dx.
	\end{equation*}
	By density, the inequality holds for any $u\in H^1_0\left(\Cont_V\right)$, and therefore $\lambda_1\left(\Cont_V\right)\ge \nu_1(V)$. The change of variable $x_1=(V/2)t$ shows that $Q_V$ is unitarily equivalent to $P_h+\pi^2$ with $h=\frac2V$, which establishes the desired result.
\end{proof}

\begin{lem} \label{lemMu1Lower}
	If $h\le4$, $\mu_1(h)\ge \frac{\pi^2h^2}{32}$.
\end{lem}

\begin{proof}
	For any $h>0$, $R_h\le P_h$, where $R_h$ is the differential operator
	\begin{equation*}
	R_h:=-h^2\frac{d^2}{dt^2}+W(t),\mbox{ with }W(t):=\left\{\begin{array}{cc}	
	0&\mbox{ if }  |t|<\sqrt{2};\\
	\pi^2&\mbox{ if } |t|\ge \sqrt{2}.\\
	\end{array}
	\right.
	\end{equation*} 
	We therefore have $\mu_1(h)\ge\xi_1(h)$, with $\xi_1(h)$ the first eigenvalue of $R_h$.
	
	The spectrum of $R_h$ is known explicitly  (it is a Schr\"odinger operator with a square well potential, studied in most textbooks on quantum mechanics, see for instance \cite[Chapter 2, Section 9]{Schiff68}). We find $\xi_1(h)=\frac{h^2}{2}\rho_1^2(h)$,
	where $\rho_1(h)$ is the smallest positive solution of the equation $\rho\tan(\rho)=\sqrt{2\pi^2/h^2-\rho^2}$. It is easily seen that the assumption $h\le 4$ implies $\rho_1(h)\ge \pi/4$, and thus	$\mu_1(h)\ge\xi_1(h)\ge \pi^2h^2/32$.
\end{proof}
Gathering all the previous estimates, we obtain, when $V\ge 1/2$, 
\[J(V)\ge \lambda_1(\Cont_V)\ge \nu_1(V)\ge \pi^2+\mu_1\left(\frac2V\right)\ge \pi^2+\frac{\pi^2}{8V^2}.\]

\section{Numerical study of the flat torus}

\subsection{Algorithm and results}
We performed in \cite{BonnaillieNoelLena2016ExpMath} a numerical study of our model problem, using the method introduced by B. Bourdin, D. Bucur and {\'E}. Oudet \cite{BouBucOud09}, with some modifications.  In their work, they looked for partitions which are optimal with respect to the sum of the eigenvalues. They passed to a relaxed formulation, looking for indicator functions instead of domains, and penalizing overlapping supports. They then discretized the resulting optimization problem, through a five points finite difference method for the Laplacian, and performed the optimization iteratively, with the projected gradient algorithm. We made the following changes to their algorithm. First, we considered general $\ell^p$-norms for the energy, rather than just the $\ell^1$-norm, in order to approach the maximum by taking a larger $p$. We also added a last step, in which we built a strong partition from the result of the optimization algorithm, and evaluated its energy without relaxation. As pointed out in \cite{BouBucOud09}, the algorithm proves to be quite sensitive to the initial condition, due to the non-convexity of the problem. For each value of $k$ and $b$, we therefore ran the algorithm several times with different initial conditions, and chose the results giving the lowest energy.

\begin{figure}[!ht]
\begin{center}

\subfigure[$b=0.7$\label{figNumb07}]{\includegraphics[width=3.5cm]{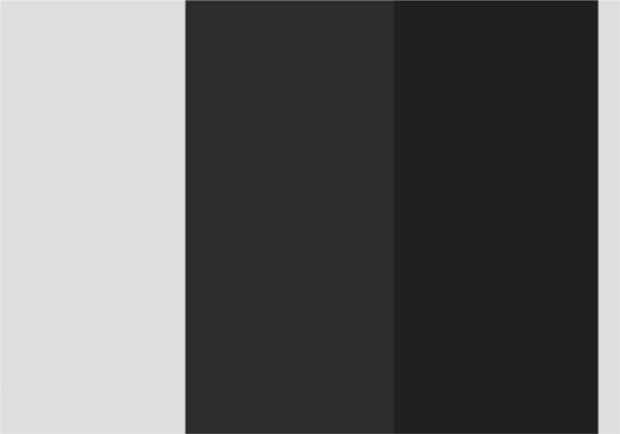}}\hfill
\subfigure[$b=0.72$\label{figNumb072}]{\includegraphics[width=3.5cm]{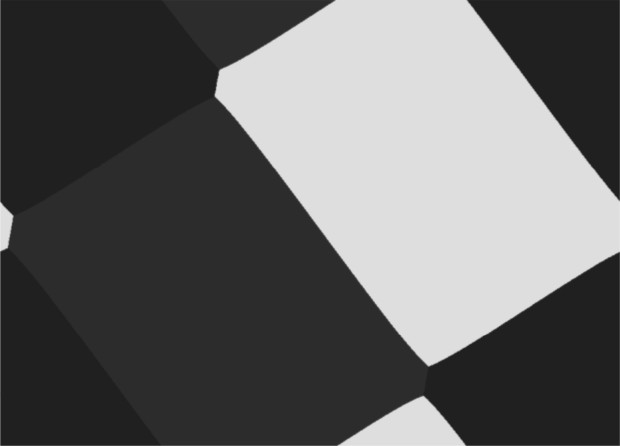}}\hfill
\subfigure[$b=1$\label{figNumb1}]{\includegraphics[width=3.5cm]{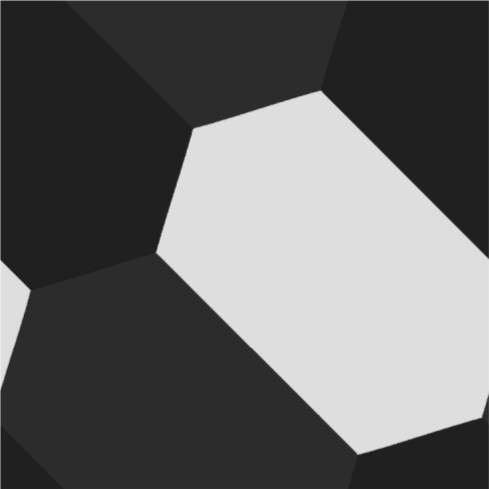}}
\end{center}
\caption{$3$-partitions for some values of $b\,$.\label{figNum}}
\end{figure}

Figure \ref{figNum} presents some  results of the numerical optimization. Comparing the partitions in Figures \ref{figNumb07} and \ref{figNumb072}, we see that $b_3$ seems close to the value $1/\sqrt{2}\simeq 0.7071$ given by Conjecture \ref{conjbk}. It appears in fact slightly higher in our numerical computations, possibly because of the approximations introduced in the algorithm. These results also suggest a transition mechanism from Figure \ref{figNumb07} to Figure \ref{figNumb072}. Indeed, we can construct a $3$-partition of $\Tor(1,1/\sqrt{2})$, with the same energy as $\mathcal D_3(1,1/\sqrt 2)$ but of a different topological type. It is represented on Figure \ref{fig3Part},  and is obtained by projecting on $\Tor(1,1/\sqrt 2)$ a nodal $6$-partition of the double covering $\Tor(2,1/\sqrt 2)$ (see \cite[Section 2.3]{BonnaillieNoelLena2016ExpMath}). The partition on Figure \ref{figNumb072} could then be obtained by a deformation which splits each singular point of order $4$ into two singular points of order $3$.
\begin{figure}[!ht]
\begin{center}
	{\includegraphics[width=3cm]{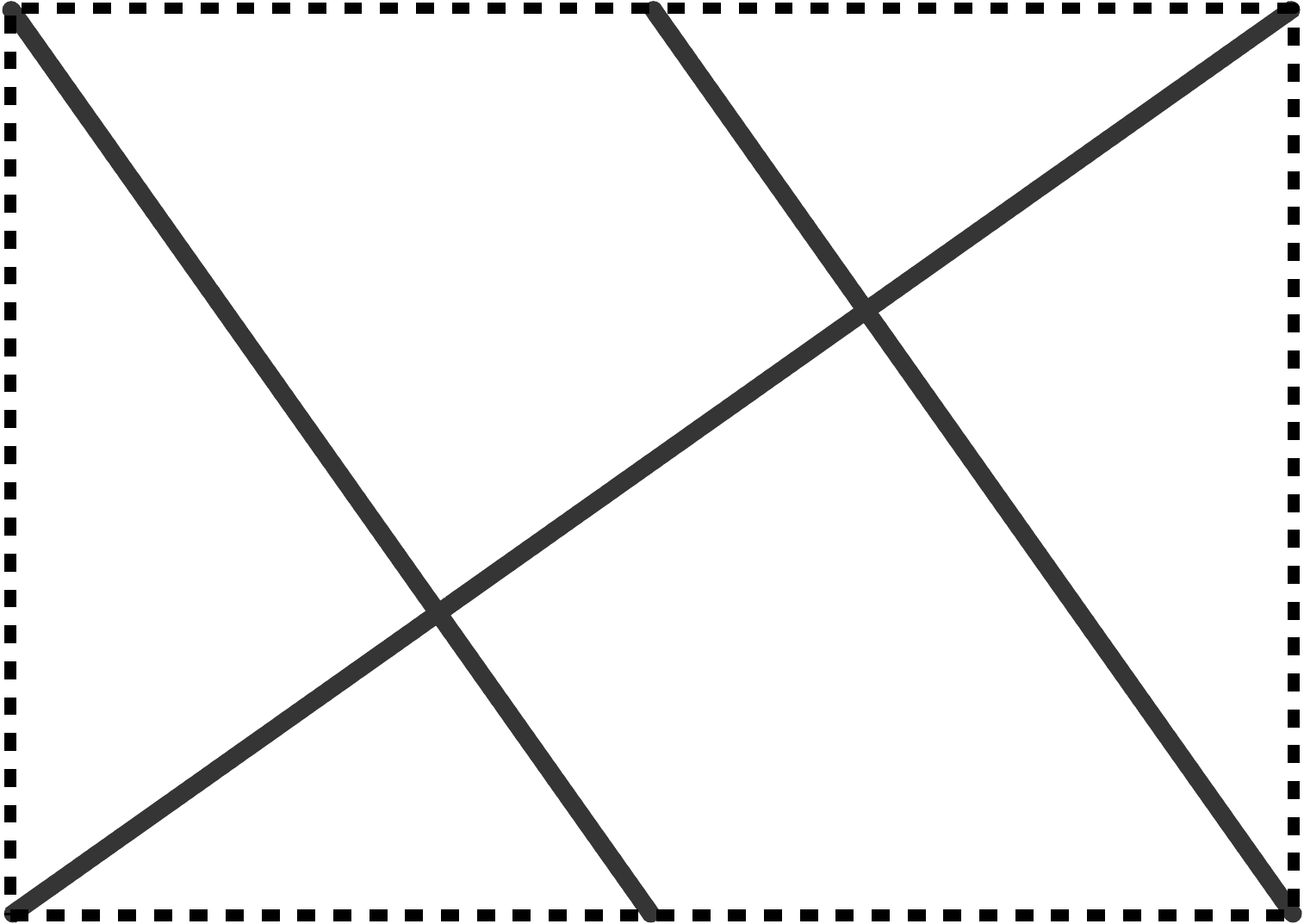}}	
\end{center}
	\caption{$3$-partition of $\Tor(1,1/\sqrt{2})$\label{fig3Part}}
\end{figure}

Finally, Figure \ref{figNumb1} strongly suggests that for $b$ quite larger than $1/\sqrt{2}$, minimal partitions of $\Tor(1,b)$ are close to hexagonal tilings. These tilings can be explicitly constructed, and their energy is an upper bound of $\mathfrak L_3(\Tor(1,b))$, smaller that $\Lambda_3(\mathcal D_3(1,b))$ for some values of $b$. Figure \ref{figUpperL3} summarizes the information thus obtained on $\mathfrak L_3(\Tor(1,b))$. The solid line represents $\lambda_1$ for a tiling hexagon, which is the energy of the hexagonal tiling, the dashed line the energy of $\mathcal D_3(1,b)$, and the crosses the results of the numerical optimization. The transition around $b=1/\sqrt 2$ clearly appears. We obtained similar results for $k\in \{4,5\}$.
\begin{figure}[!ht]
\begin{center}
	{\includegraphics[width=7cm]{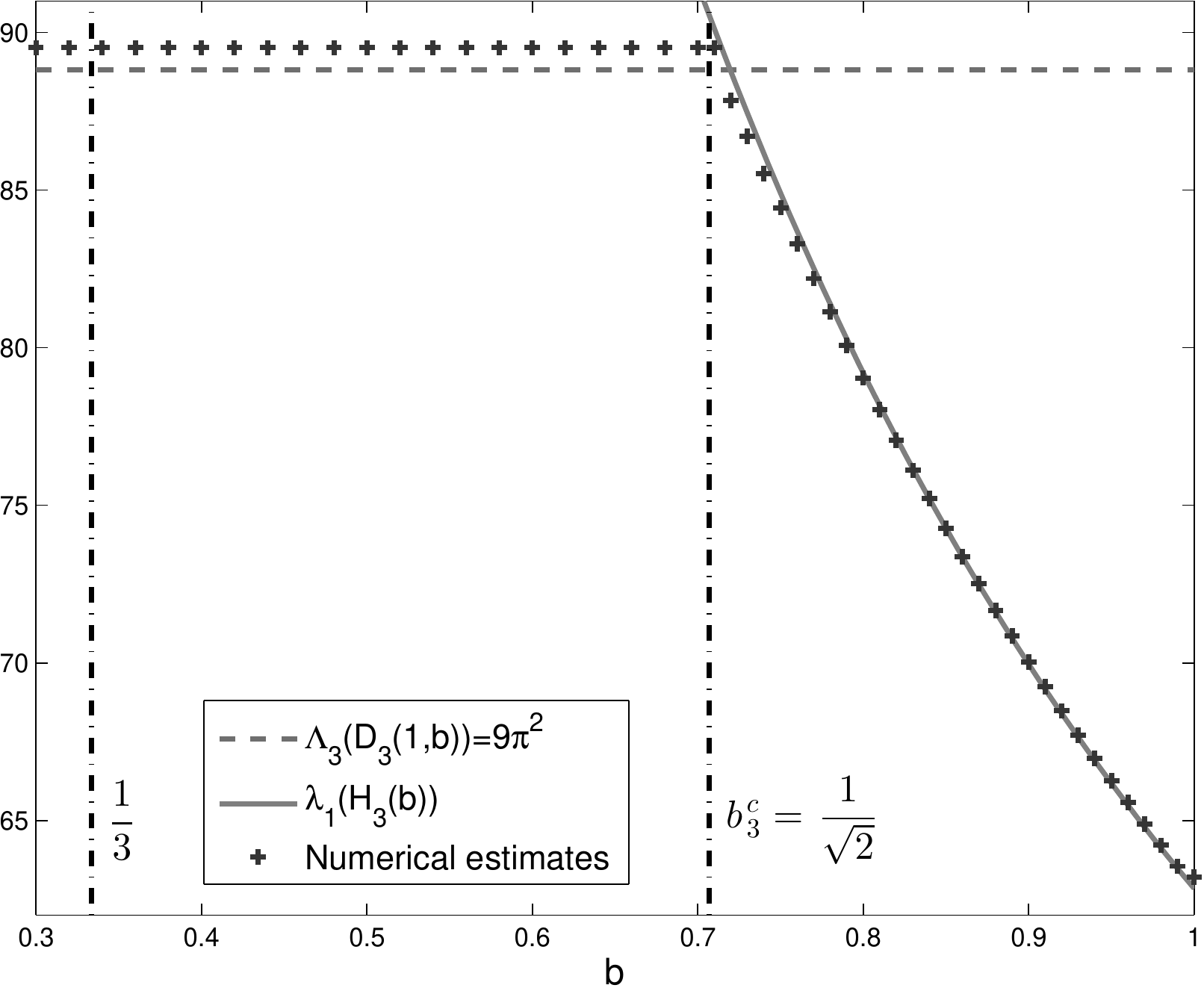}}	
\end{center}
	\caption{Upper bounds of $\mathfrak L_3(\Tor(1,b))$\label{figUpperL3}}
\end{figure}

\subsection{Tilings}

We constructed explicitly hexagonal tilings of the same topological type as the numerical results and satisfying the equal angle meeting property \cite[Section 4]{BonnaillieNoelLena2016ExpMath}. The results are summarized in the following theorem, from \cite[Section 1.3]{BonnaillieNoelLena2016ExpMath}.
\begin{thm}
\label{thmExkPart}
For $k\in\{3\,,\,4\,,\,5\}\,$, there exists $\bH_k \in (0,1)$ such that, for any $b \in (\bH_k,1]\,$, there exists a tiling of $\Tor(1,b)$ by $k$ hexagons that satisfies the equal angle meeting property. We denote by $\Hex_k(b)$ the corresponding tiling domain, and we have
\[\mathfrak{L}_k(\Tor(1,b))\le \min\left(k^2\pi^2, \lambda_1(\Hex_k(b))\right)
,\quad \forall b\in(\bH_{k},1] \,.\]
More explicitly, we can choose 
\[\bH_3=\frac{\sqrt{11}-\sqrt{3}}{4}\simeq 0.396\,,\quad \bH_4=\frac{1}{2\sqrt{3}}\simeq 0.289< b_{4}=\frac 12\,,\quad\mbox{and }\quad \bH_5=\frac{\sqrt{291}-5\sqrt{3}}{36}\simeq 0.233\,.\]
\end{thm}
In order to test the minimality of these tilings, we used the pair compatibility condition (see \cite[Section 4.5]{BonnaillieNoelLena2016ExpMath}). Indeed, if one of these tilings is a minimal $k$-partition of $\Tor(1,b)$, Corollary \ref{corPCC} implies that $\lambda_1(\Hex_k(b))=\lambda_2(2\Hex_k(b))$, with $2\Hex_k(b)$ any one of the polygonal domains obtained by gluing two copies of $\Hex_k(b)$ along corresponding sides. Numerically, this condition does not seem to be met for $b$ close to $1/\sqrt{2}$ when $k=3$, to $1/2$ when $k=4$, and to $1/\sqrt{6}$ and $1$ when $k=5$. Hexagonal tilings therefore appear not to be minimal under these conditions. This idea is supported by the numerical values of the energy, and by the slight curvature visible in the boundary of the numerically obtained partitions. Let us finally point out that when $k=5$ and $b=1$, the numerical result is very close to the partition into $5$ squares represented on Figure \ref{fig5Part} (see \cite[Section 4.4]{BonnaillieNoelLena2016ExpMath}). 
\begin{figure}[!ht]
\begin{center}
	{\includegraphics[width=3cm]{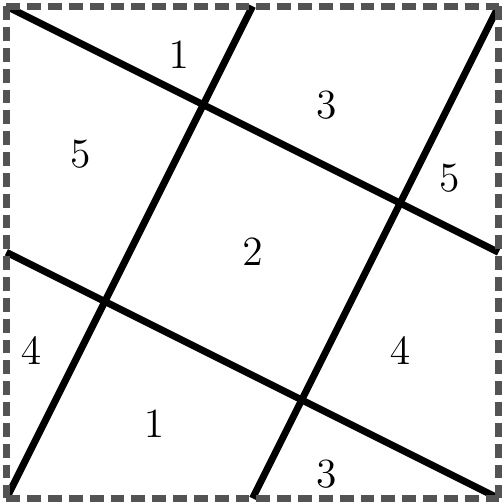}}	
\end{center}
	\caption{$5$-partition of $\Tor(1,1)$\label{fig5Part}}
\end{figure}

These numerical findings reveal a rich structure for minimal partitions of the flat rectangular torus. A better understanding would however require faster numerical algorithms and new theoretical methods.

\subsection*{Acknowledgments}
The author thanks the organizers of the Bru-To PDE's Conference, held at the University of Turin in May 2016, for inviting him to give a talk, on which the present paper is based. This work was partially supported by the ANR (Agence Nationale de la Recherche), project OPTIFORM n$^\circ$ ANR-12-BS01-0007-02, and by the ERC, project COMPAT, ERC-2013-ADG n$^\circ$ 339958. 

{\small
\bibliographystyle{plain}

}

\end{document}